\documentclass[11pt]{amsart}

\usepackage[T1]{fontenc}
\usepackage[utf8]{inputenc}
\usepackage{lmodern}
\usepackage{microtype}
\usepackage[margin=1in]{geometry}
\usepackage{amsmath,amssymb,amsthm,mathtools}
\usepackage{mathrsfs}
\usepackage{enumitem}
\usepackage{xcolor}
\usepackage[colorlinks=true,linkcolor=blue,citecolor=blue,urlcolor=blue]{hyperref}
\usepackage[nameinlink,noabbrev]{cleveref}

\crefname{theorem}{theorem}{theorems}
\Crefname{theorem}{Theorem}{Theorems}
\crefname{proposition}{proposition}{propositions}
\Crefname{proposition}{Proposition}{Propositions}
\crefname{lemma}{lemma}{lemmas}
\Crefname{lemma}{Lemma}{Lemmas}
\crefname{corollary}{corollary}{corollaries}
\Crefname{corollary}{Corollary}{Corollaries}
\crefname{definition}{definition}{definitions}
\Crefname{definition}{Definition}{Definitions}
\crefname{remark}{remark}{remarks}
\Crefname{remark}{Remark}{Remarks}
\crefname{question}{question}{questions}
\Crefname{question}{Question}{Questions}

\newtheorem{theorem}{Theorem}[section]
\newtheorem{proposition}[theorem]{Proposition}
\newtheorem{lemma}[theorem]{Lemma}

\theoremstyle{definition}
\newtheorem{definition}[theorem]{Definition}
\newtheorem{remark}[theorem]{Remark}

\newcommand{\ZF}{\mathrm{ZF}}
\newcommand{\DC}{\mathrm{DC}}
\newcommand{\BPI}{\mathrm{BPI}}

\newcommand{\ACfin}{\mathrm{AC}_{\mathrm{fin}}}
\newcommand{\Cn}{\mathrm{C}_n}
\newcommand{\Ctwo}{\mathrm{C}_2}
\newcommand{\Cthree}{\mathrm{C}_3}
\newcommand{\HS}{\operatorname{HS}}
\newcommand{\Fix}{\operatorname{Fix}}
\newcommand{\Pow}{\mathcal P}
\newcommand{\Fin}{\mathrm{Fin}}
\newcommand{\omegaone}{\omega_1}
\newcommand{\Succ}{\operatorname{Succ}}
\newcommand{\clrho}{\operatorname{cl}_{\rho}}
\newcommand{\Star}{\operatorname{Star}}
\newcommand{\suppnode}{\operatorname{supp}_{\mathrm{node}}}
\newcommand{\Ftwo}{\mathbb F_2}
\newcommand{\forces}{\Vdash}
\newcommand{\Gcas}{\mathscr G^{\mathrm{1cas}}_{\rho}}
\newcommand{\Fcas}{\mathscr F^{\mathrm{1cas}}_{\rho}}
\newcommand{\Scas}{\mathcal S^{\mathrm{1cas}}_{\rho}}
\newcommand{\Ncas}{\mathcal N^{\mathrm{1cas}}_{\rho}}

\title[A One-Step Cascade Symmetric Model]{A One-Step Cascade Symmetric Model:\
Rank-$1$ Packets, Binary Shielding, and the Even Exact-Cardinality Profile}

\author{Frank Gilson}
\date{March 26, 2026}

\begin{document}

\begin{abstract}
We introduce a one-step cascade symmetric system whose local symmetry geometry is organized by finite $\rho$-closed windows and one-step stars rather than by rowwise-independent toggles. The resulting symmetric model isolates a new $\ZF+\DC+\neg\BPI$ geometry in which rank-$1$ hereditarily symmetric reals admit a packet normalization theorem over countable $\rho$-closed supports.

The technical center of the paper is the finite star-span lemma and the associated rank-$1$ packet calculus. From this we obtain a normalization theorem and a two-layer coding consequence for rank-$1$ reals (in the metatheory, via a well-orderable base of packets). We then apply the same binary fresh-support shielding pattern to prove $\neg\Ctwo$, hence $\neg\ACfin$, and therefore the failure of every even $\Cn$ (where $\Cn$ denotes the principle that every family of nonempty $n$-element sets admits a choice function). On the odd side, the present bounded packet calculus remains dyadic: support-fixed local actions factor through finite $2$-groups, bounded support-equivariant quotients of finite local orbits have power-of-two size, and trace-separated bounded rigid ternary families admit canonical selectors within a fixed finite trace window. Accordingly, the odd exact-cardinality profile remains open beyond the current local binary machinery.
\end{abstract}

\maketitle

\section{Introduction}

Symmetric extensions are a central source of choiceless models, but the geometry of the symmetry can be as important as the forcing itself. The purpose of this paper is to isolate a specific one-step cascade geometry in which the usable support analysis is governed by finite $\rho$-closed windows and one-step stars rather than by rowwise-independent toggles. This geometry supports three distinct kinds of statements: a direct proof that the Boolean Prime Ideal Theorem ($\BPI$) fails, a finite-window linear algebra for rank-$1$ names, and an exact-cardinality failure theorem on the even side. The relevant structural background comes from the classical theory of intermediate models and symmetric extensions, together with the modern framework of symmetric systems and their equivalences; see, for example, \cite{Grigorieff1975,Karagila2019Iterating,KaragilaSchilhan2026Theory}.

The technical center of the paper is the finite star-span lemma (\Cref{lem:starspan}). It gives the correct local replacement for rowwise-independent toggles and leads to a packet calculus for rank-$1$ hereditarily symmetric reals, by which we mean hereditarily symmetric $\mathbb P_1$-names for subsets of $\omega$. That packet calculus yields the normalization theorem of \Cref{thm:rank1norm} and the two-layer coding consequence of \Cref{thm:twolayer}, placing the model in a natural Kinna--Wagner context; compare \cite{KaragilaSchilhan2026IMKWP}.

The weak-choice consequences split into two sharply different parts. On the even side, the binary shielding mechanism gives a direct proof of $\neg\Ctwo$, hence $\neg\ACfin$, and then of $\neg\Cn$ for every even $n\geq 2$ by the divisibility propagation of \Cref{prop:div}. On the odd side, beginning with $\Cthree$, the current local machinery yields only boundary theorems: local actions remain $2$-group actions, bounded support-equivariant quotients remain dyadic, and trace-separated bounded rigid ternary families admit canonical selectors within a fixed finite trace window. Thus the paper proves a genuine even exact-cardinality profile while treating the odd side as an explicitly open boundary. Standard background for these choice principles may be found in \cite{Jech1973AC,Jech2003,HowardRubin1998}.

Our main theorem may be summarized as follows.

\begin{theorem}[Main theorem]\label{thm:main}
Let $\Ncas$ be the one-step cascade symmetric model defined in \Cref{sec:model}. Then:
\begin{enumerate}[label=\textup{(\alph*)},leftmargin=2.4em]
    \item $\Ncas\models \ZF+\DC+\neg\BPI$.
    \item Every rank-$1$ hereditarily symmetric real in $\Ncas$ normalizes to a countable packet scheme over a countable $\rho$-closed support, and hence admits, in the metatheory via a well-orderable base of packets, a two-layer code by such a support together with a countable family of exact finite cascade packets.
    \item $\Ncas\models \neg\Ctwo$, hence $\Ncas\models \neg\ACfin$.
    \item For every even $n\geq 2$, $\Ncas\models \neg\Cn$.
\end{enumerate}
In addition, \Cref{prop:nonconj} records the geometric fact that $\Gcas$ contains elements with stationary row-support and therefore with uncountable total support, while \Cref{sec:odd} explains why the present bounded packet calculus leaves the odd side beginning with $\Cthree$ open.
\end{theorem}

The organization of the paper is as follows. \Cref{sec:model} defines the one-step cascade symmetric system, proves the base theorem, and records the intrinsic geometric observation that the cascade group has elements with stationary row-support and uncountable total support. \Cref{sec:starspan} proves the finite star-span lemma. \Cref{sec:rank1} develops the rank-$1$ packet calculus, normalization theorem, and two-layer coding consequence. \Cref{sec:even} proves the even exact-cardinality theorem bank. \Cref{sec:odd} records the odd-side structural boundary. \Cref{sec:open} collects the main questions left open by the present packet calculus. \Cref{sec:related} situates the construction relative to the recent theory of symmetric systems, dependent choice in symmetric extensions, and intermediate-model/Kinna--Wagner results.

\section{The one-step cascade system and the base theorem}\label{sec:model}

We begin with the Stage~0 predecessor map and the Stage~1 forcing. The key feature is that the automorphisms toggle one row together with the matching rows above its immediate successors in the predecessor forest. The point is to retain the fresh-support contradiction pattern used in the anti-$\BPI$ argument while replacing arbitrary countable row sets by countable $\rho$-closed node sets.

\begin{definition}[Stage~0 forcing and successor notation]
Let $\mathbb P_0$ be the forcing of finite partial regressive functions on $\omegaone$. If $G_0\subseteq\mathbb P_0$ is $V$-generic, write
\[
\rho:\omegaone\setminus\{0\}\to\omegaone
\]
for the induced total regressive predecessor map, write
\[
\Succ_\rho(\xi)=\{\eta<\omegaone:\rho(\eta)=\xi\},
\]
and write $\clrho(A)$ for the $\rho$-closure of $A\subseteq\omegaone$.
\end{definition}

\begin{definition}[Stage~1 forcing and one-step cascade generators]
In $V[G_0]$, let
\[
\mathbb P_1=\operatorname{Fn}(\omegaone\times\omega\times\omega,2,{<}\omega),
\]
with generic reals $c_{\xi,i}\in 2^\omega$ for $\xi<\omegaone$ and $i<\omega$. For $\xi<\omegaone$, $i<\omega$, and $s\subseteq\omega$ which is finite or cofinite, let
\[
\tau^{\mathrm{1cas}}_{\xi,i,s}
\]
denote the automorphism which toggles the row $(\xi,i)$ along $s$, toggles every row $(\eta,i)$ with $\eta\in\Succ_\rho(\xi)$ along the same set $s$, and fixes all other rows. Let $\Gcas$ be the group generated by these automorphisms.
\end{definition}

\begin{definition}[The symmetric system]
For $A\subseteq\omegaone$, let $\Fix^{\mathrm{1cas}}_{\rho}(A)$ be the subgroup of $\Gcas$ acting trivially on every row above $\clrho(A)$. Let $\Fcas$ be the filter of subgroups of $\Gcas$ generated by the subgroups $\Fix^{\mathrm{1cas}}_{\rho}(A)$ for countable $A\subseteq\omegaone$. We write
\[
\Scas=(\mathbb P_1,\Gcas,\Fcas)
\]
for the resulting symmetric system.
\end{definition}

\begin{remark}
Allowing cofinite toggle sets is needed for the complement-flip automorphisms used in the anti-$\BPI$ and no-selector arguments below. The finite local linear algebra of \Cref{lem:starspan} is unchanged, since it uses only the single-bit generators $\tau^{\mathrm{1cas}}_{\xi,i,\{n\}}$.
\end{remark}
\begin{remark}\label{rem:ccc}
Both $\mathbb P_0$ and $\mathbb P_1$ satisfy the countable chain condition. For $\mathbb P_1=\operatorname{Fn}(\omegaone\times\omega\times\omega,2,{<}\omega)$ this is the usual $\Delta$-system argument for finite-support Cohen forcing: an uncountable antichain would have an uncountable subfamily with pairwise disjoint finite domains, hence pairwise compatible conditions, a contradiction. The same argument applies to the finite partial regressive forcing $\mathbb P_0$.
\end{remark}

\begin{lemma}[Fresh one-step separation]\label{lem:fresh-separation}
Let $A\subseteq\omegaone$ be countable and $\rho$-closed. Then there exist ordinals $\beta,\gamma<\omegaone$ such that
\[
\beta,\gamma\notin A,
\qquad
\beta\neq\gamma,
\qquad
\gamma\notin\Succ_\rho(\beta),
\qquad\text{and}\qquad
\Succ_\rho(\beta)\cap A=\varnothing.
\]
\end{lemma}

\begin{proof}
The set $A\cup\rho[A]$ is countable. Choose distinct ordinals $\gamma<\beta<\omegaone$ outside $A\cup\rho[A]$. Since $\beta\notin\rho[A]$, there is no $\eta\in A$ with $\rho(\eta)=\beta$, so $\Succ_\rho(\beta)\cap A=\varnothing$. Also every element of $\Succ_\rho(\beta)$ is strictly above $\beta$, because $\rho$ is regressive on $\omegaone\setminus\{0\}$. Since $\gamma<\beta$, it follows that $\gamma\notin\Succ_\rho(\beta)$.
\end{proof}

\begin{lemma}[Finite shielding]\label{lem:finite-shielding}
Let $q\in\mathbb P_1$, let $\beta<\omegaone$ and $i<\omega$, and define
\[
C(q,\beta,i)=\bigl\{n<\omega : q(\beta,i,n)\text{ is defined, or }q(\eta,i,n)\text{ is defined for some }\eta\in\Succ_\rho(\beta)\bigr\}.
\]
Then $C(q,\beta,i)$ is finite. If $s\subseteq\omega$ is finite or cofinite and $s\cap C(q,\beta,i)=\varnothing$, then
\[
\tau^{\mathrm{1cas}}_{\beta,i,s}(q)=q.
\]
\end{lemma}

\begin{proof}
Because $q$ is a finite condition, only finitely many triples $(\xi,j,n)$ lie in its domain, so the set $C(q,\beta,i)$ is finite. If $s\cap C(q,\beta,i)=\varnothing$, then the generator $\tau^{\mathrm{1cas}}_{\beta,i,s}$ does not toggle any domain entry of $q$: by definition, the only potentially affected coordinates are those on the row $(\beta,i)$ and on the rows $(\eta,i)$ with $\eta\in\Succ_\rho(\beta)$, and at each such coordinate the relevant bit index $n$ lies outside $s$. Hence every value of $q$ is preserved, so $\tau^{\mathrm{1cas}}_{\beta,i,s}(q)=q$.
\end{proof}

\begin{proposition}[Normality of the support filter]\label{prop:filter-normal}
The filter $\Fcas$ is a normal filter of subgroups of $\Gcas$.
\end{proposition}

\begin{proof}
Each basic generator $\tau^{\mathrm{1cas}}_{\xi,i,s}$ acts on each bit value $q(\zeta,j,n)$ by XOR with $0$ or $1$; since $\Ftwo$-addition is commutative and associative, any two generators commute and $\Gcas$ is abelian. In particular, every subgroup of $\Gcas$ is normal. Since each subgroup $\Fix^{\mathrm{1cas}}_{\rho}(A)$ is a subgroup of $\Gcas$, every conjugate of such a subgroup is itself, and the filter generated by these subgroups is closed under conjugation. Thus $\Fcas$ is a normal filter of subgroups.
\end{proof}

\begin{proposition}[$\omega_1$-completeness of the support filter]\label{prop:filter-omega1}
The filter $\Fcas$ is $\omega_1$-complete. More precisely, if $\{A_n:n<\omega\}$ is a countable family of countable $\rho$-closed subsets of $\omegaone$, then for
\[
A=\bigcup_{n<\omega} A_n
\]
one has that $A$ is countable and $\rho$-closed, and
\[
\Fix^{\mathrm{1cas}}_{\rho}(A)\subseteq \bigcap_{n<\omega}\Fix^{\mathrm{1cas}}_{\rho}(A_n).
\]
\end{proposition}

\begin{proof}
The union $A$ is countable. It is also $\rho$-closed: if $\eta\in A$, then $\eta\in A_n$ for some $n<\omega$, and because $A_n$ is $\rho$-closed, every predecessor obtained by iterating $\rho$ from $\eta$ still lies in $A_n\subseteq A$. If an automorphism acts trivially on every row above $A$, then it acts trivially on every row above each $A_n$, so
\[
\Fix^{\mathrm{1cas}}_{\rho}(A)\subseteq \Fix^{\mathrm{1cas}}_{\rho}(A_n)
\qquad\text{for all }n<\omega.
\]
Thus $\Fix^{\mathrm{1cas}}_{\rho}(A)$ is contained in the intersection, and the generated filter is $\omega_1$-complete.
\end{proof}

\begin{theorem}[Base theorem]\label{thm:base}
Let $H_1\subseteq\mathbb P_1$ be $V[G_0]$-generic, and let
\[
\Ncas=\HS^{H_1}_{\Scas}
\]
be the associated symmetric extension. Then $\Ncas\models \ZF+\DC+\neg\BPI$.
\end{theorem}

\begin{proof}
By the standard symmetric extension theorem, $\Ncas$ is a transitive model of $\ZF$ between $V[G_0]$ and $V[G_0][H_1]$; see \cite[Chapter 15]{Jech2003} and the streamlined formulation in \cite{KaragilaSchilhan2026Theory}. The substantive point is the failure of $\BPI$.

Fix a row number $i<\omega$. In $V[G_0][H_1]$, let $\mathbb T^i_{\rho}$ be the Boolean subalgebra of $\Pow(\omega)/\Fin$ generated by the classes
\[
X^i_{\beta,\gamma}=[E^i_{\beta,\gamma}]_{\Fin},
\qquad
E^i_{\beta,\gamma}=\{n<\omega:c_{\beta,i}(n)=c_{\gamma,i}(n)\},
\]
for $\beta\neq\gamma$. We claim that no hereditarily symmetric ultrafilter on $\mathbb T^i_{\rho}$ exists.

Suppose toward a contradiction that some condition $p\in\mathbb P_1$ forces that $\dot U$ is an ultrafilter on $\mathbb T^i_{\rho}$, and let $A\subseteq\omegaone$ be a countable $\rho$-closed support for $\dot U$. By \Cref{lem:fresh-separation}, choose $\beta,\gamma\notin A$ such that $\beta\neq\gamma$, $\gamma\notin\Succ_\rho(\beta)$, and no immediate successor of $\beta$ lies in $A$. Let
\[
X=X^i_{\beta,\gamma}\in\mathbb T^i_{\rho}.
\]
Since $p$ forces that $\dot U$ is an ultrafilter on the whole algebra, there is $q\le p$ deciding whether $X\in\dot U$.

Let $C=C(q,\beta,i)\subseteq\omega$ be as in \Cref{lem:finite-shielding}. Choose a cofinite ground-model set $s\subseteq\omega\setminus C$, and let
\[
\tau=\tau^{\mathrm{1cas}}_{\beta,i,s}.
\]
By \Cref{lem:finite-shielding}, $\tau(q)=q$. By the choice of $\beta$, the automorphism $\tau$ acts trivially on every row over $A$, so $\tau\in\Fix^{\mathrm{1cas}}_{\rho}(A)$ and therefore $\tau(\dot U)=\dot U$.

Now $\tau$ toggles the row $(\beta,i)$ on the cofinite set $s$ and leaves the row $(\gamma,i)$ untouched, because $\gamma\neq\beta$ and $\gamma\notin\Succ_\rho(\beta)$. Hence away from the finite complement $\omega\setminus s$, equality between $c_{\beta,i}$ and $c_{\gamma,i}$ is turned into inequality. Thus
\[
\tau(X)=\neg X
\qquad\text{in }\Pow(\omega)/\Fin.
\]

If $q\forces X\in\dot U$, then applying $\tau$ yields
\[
q=\tau(q)\forces \tau(X)\in\tau(\dot U),
\]
so $q\forces \neg X\in\dot U$ as well, contradicting ultrafilterhood. If instead $q\forces X\notin\dot U$, then because $q$ also forces that $\dot U$ is an ultrafilter, it forces $\neg X\in\dot U$; applying $\tau$ again gives $q\forces X\in\dot U$, the same contradiction. Therefore no supported ultrafilter name on $\mathbb T^i_{\rho}$ exists. So $\Ncas\models\neg\BPI$.

By \Cref{prop:filter-normal,prop:filter-omega1}, the filter $\Fcas$ is normal and $\omega_1$-complete. Since $\mathbb P_1$ is ccc, the standard preservation theorem for dependent choice in symmetric extensions applies; see \cite{Pincus1977,Karagila2019PreservingDC,Karagila2019Iterating}. Hence $\Ncas\models \ZF+\DC+\neg\BPI$.
\end{proof}

\begin{proposition}[Uncountable support in the one-step cascade geometry]\label{prop:nonconj}
The group $\Gcas$ contains elements with stationary, hence uncountable, row-support, and therefore elements with uncountable total support. Consequently, via any coordinate relabeling of the forcing, the cascade system is not conjugate to any symmetric system on $\mathbb P_1$ whose automorphism group consists entirely of elements with countable total support.
\end{proposition}

\begin{proof}
Since $\rho$ is a regressive function on $\omegaone\setminus\{0\}$, which is a club subset of $\omegaone$, Fodor's lemma yields $\xi<\omegaone$ such that the fiber $\Succ_\rho(\xi)$ is stationary. Here we use that $\mathbb P_0$ is ccc, hence preserves $\omega_1$.

Fix $i<\omega$ and a nonempty finite $s\subseteq\omega$. Then the generator $\tau^{\mathrm{1cas}}_{\xi,i,s}$ acts on the row $(\xi,i)$ and on every row $(\eta,i)$ with $\eta\in\Succ_\rho(\xi)$. Its row-support is therefore
\[
\{\xi\}\cup \Succ_\rho(\xi),
\]
which is stationary and in particular uncountable.

A coordinate relabeling means a bijection of the index set $\omegaone\times\omega\times\omega$, and such a relabeling preserves the cardinality of the total support of every group element, namely the set of triples $(\eta,j,n)$ on which it acts nontrivially. Therefore any symmetric system on $\mathbb P_1$ whose automorphism group consists entirely of elements with countable total support remains of that form after coordinate relabeling. The cascade generator above has uncountable total support. Hence no coordinate relabeling can conjugate the cascade system to any symmetric system with this countable-total-support property.
\end{proof}

\section{Finite $\rho$-closed windows and the star-span lemma}\label{sec:starspan}

The usable local geometry of the one-step cascade model is not arbitrary rowwise independence, but completeness on finite $\rho$-closed windows. This is where the predecessor map $\rho$ becomes structurally productive.

\begin{definition}[$\rho$-closed windows]
A set $K\subseteq\omegaone$ is \emph{$\rho$-closed} if $\clrho(K)=K$. A finite $\rho$-closed set will be called a \emph{finite $\rho$-window}.
\end{definition}

\begin{definition}[Finite one-step stars]
For a finite $\rho$-window $K\subseteq\omegaone$ and $\xi\in K$, define
\[
\Star_K(\xi)=\{\xi\}\cup (\Succ_{\rho}(\xi)\cap K).
\]
The associated star vector is the characteristic vector $\mathbf 1_{\Star_K(\xi)}\in \Ftwo^K$.
\end{definition}

\begin{lemma}[Finite star-span lemma]\label{lem:starspan}
Let $K\subseteq\omegaone$ be a finite $\rho$-window. Then the family
\[
\{\mathbf 1_{\Star_K(\xi)}:\xi\in K\}
\]
spans $\Ftwo^K$. Equivalently, every finite bit-difference pattern on the rows over $K$ is realized by a finite composition of one-step cascade generators supported inside $K$.
\end{lemma}

\begin{proof}
Order $K$ by decreasing height in the finite predecessor forest determined by $\rho$, so that every child appears before its parent. Form the matrix whose $\xi$-th column is the star vector $\mathbf 1_{\Star_K(\xi)}\in\Ftwo^K$ written in this order.

Each column has a $1$ in the row indexed by $\xi$, because $\xi\in\Star_K(\xi)$. If $\eta\in\Succ_{\rho}(\xi)\cap K$, then $\eta$ appears earlier than $\xi$ in the chosen order, so all other $1$'s in the $\xi$-th column occur strictly above the diagonal entry. Thus the star-incidence matrix is upper triangular with diagonal entries equal to $1$. Over $\Ftwo$ such a matrix is invertible, so the star vectors are linearly independent. Since there are $|K|$ of them in the $|K|$-dimensional space $\Ftwo^K$, they form a basis and therefore span all finite bit-difference patterns on $K$.
\end{proof}

\begin{remark}
\Cref{lem:starspan} is the exact local replacement for rowwise-independent toggles, and it is the combinatorial input reused throughout the rank-$1$ analysis.
\end{remark}

\section{\texorpdfstring{Rank-$1$}{Rank-1} packet calculus and normalization}\label{sec:rank1}

We now convert the finite-window linear algebra into a language for rank-$1$ hereditarily symmetric names. By a \emph{rank-$1$ hereditarily symmetric real} we mean a hereditarily symmetric $\mathbb P_1$-name for a subset of $\omega$---that is, a name $\dot x$ all of whose members are pairs $(\check m,r)$ with $m\in\omega$ and $r\in\mathbb P_1$; equivalently, an element of $\Ncas$ that is a subset of $\omega$.

\begin{definition}[Exact finite cascade packet]
An \emph{exact finite cascade packet} is a finite condition
\[
r\in\mathbb P_1
\]
whose node-support
\[
\suppnode(r)=\{\xi<\omegaone:\exists i,n\; r(\xi,i,n)\text{ is defined}\}
\]
is finite and $\rho$-closed.
\end{definition}

\begin{definition}[Rank-$1$ packet scheme]
A \emph{rank-$1$ packet scheme} is a pair
\[
\mathfrak S=(A,\langle \mathcal C_m:m<\omega\rangle)
\]
such that:
\begin{enumerate}[label=\textup{(\arabic*)},leftmargin=2.4em]
    \item $A\subseteq\omegaone$ is countable and $\rho$-closed;
    \item for each $m<\omega$, $\mathcal C_m$ is a countable family of exact finite cascade packets $r$ with $\suppnode(r)\subseteq A$.
\end{enumerate}
The associated name is
\[
\dot x_{\mathfrak S}=\{(\check m,r):m<\omega,\ r\in \mathcal C_m\}.
\]
\end{definition}

\begin{proposition}\label{prop:packet-hs}
Every packet name $\dot x_{\mathfrak S}$ arising from a rank-$1$ packet scheme $\mathfrak S=(A,\langle \mathcal C_m:m<\omega\rangle)$ is hereditarily symmetric and is supported by $A$.
\end{proposition}

\begin{proof}
Every packet $r\in\mathcal C_m$ has node-support contained in $A$. Hence every automorphism in $\Fix^{\mathrm{1cas}}_{\rho}(A)$ fixes each packet appearing in $\dot x_{\mathfrak S}$, and therefore fixes each pair $(\check m,r)$ in the name. Thus $\Fix^{\mathrm{1cas}}_{\rho}(A)$ fixes $\dot x_{\mathfrak S}$ pointwise, so $\dot x_{\mathfrak S}\in\HS$ and is supported by $A$.
\end{proof}

\begin{lemma}[Decision invariance on $\rho$-closed support]\label{lem:decision}
Let $\dot x\in\HS$ be supported by a countable $\rho$-closed set $A$. If $p$ decides whether $\check m\in\dot x$, then the restriction of $p$ to the rows over $A$ decides the same truth value.
\end{lemma}

\begin{proof}
Let $r$ be the restriction of $p$ to the rows over $A$. Suppose toward a contradiction that $p\forces \check m\in\dot x$ but $r$ does not decide the same truth value. Then some condition $q\le r$ forces $\check m\notin\dot x$. The conditions $p$ and $q$ agree on every coordinate over $A$.

Refine $p$ and $q$ to common-domain conditions $p',q'$ by padding undefined coordinates with $0$ on the finite union of their domains. Let $K$ be the finite $\rho$-closure of the set of nodes outside $A$ mentioned by either $p'$ or $q'$. Then $K$ is a finite $\rho$-window disjoint from $A$.

Fix a local row number $i$ and bit coordinate $n$ appearing in this common domain. The difference between the $K$-patterns of $p'$ and $q'$ at $(i,n)$ is a vector in $\Ftwo^K$. By \Cref{lem:starspan}, that vector is realized by a finite product of one-step cascade generators supported inside $K$. Repeating this for each relevant pair $(i,n)$, we obtain a family of automorphisms supported inside $K$. Automorphisms attached to different pairs $(i,n)$ act on disjoint bit coordinates of any finite condition, so they commute and compose without interference. Multiplying them, we obtain
\[
\pi\in\Fix^{\mathrm{1cas}}_{\rho}(A)
\]
such that $\pi(p')=q'$.

Because $A$ supports $\dot x$, one has $\pi(\dot x)=\dot x$. Applying the symmetry lemma for the forcing relation yields
\[
q'=\pi(p')\forces \check m\in \pi(\dot x)=\dot x,
\]
contradicting that $q'$ forces $\check m\notin\dot x$. The case in which $p\forces \check m\notin\dot x$ is identical, using a witness $q\le r$ with $q\forces \check m\in\dot x$ and exchanging the roles of membership and non-membership. Therefore $r$ already decides the same truth value as $p$.
\end{proof}

\begin{theorem}[Rank-$1$ normalization]\label{thm:rank1norm}
Let $\dot x\in\HS$ be a name for a subset of $\omega$, supported by a countable $\rho$-closed set $A\subseteq\omegaone$. Then there exists a rank-$1$ packet scheme $\mathfrak S$ over $A$ such that
\[
\forces \dot x=\dot x_{\mathfrak S}.
\]
\end{theorem}

\begin{proof}
Fix $m<\omega$. By \Cref{rem:ccc}, $\mathbb P_1$ is ccc. Choose a countable maximal antichain $D_m$ deciding whether $\check m\in\dot x$. Let
\[
D_m^+=\{p\in D_m:p\forces \check m\in\dot x\}.
\]
For each $p\in D_m^+$, let $r_p$ be the restriction of $p$ to the rows over $A$. By \Cref{lem:decision}, the condition $r_p$ still forces $\check m\in\dot x$. Because $p$ is finite and $A$ is $\rho$-closed, the node-support of $r_p$ is finite and $\rho$-closed, so $r_p$ is an exact finite cascade packet over $A$.

Define
\[
\mathcal C_m=\{r_p:p\in D_m^+\},
\qquad
\mathfrak S=(A,\langle\mathcal C_m:m<\omega\rangle).
\]
Each $\mathcal C_m$ is countable, so $\mathfrak S$ is a rank-$1$ packet scheme. Let $\dot y=\dot x_{\mathfrak S}$.

We claim that $\forces \dot x=\dot y$. If $m\in \dot x[G]$ in some generic extension, then there is $p\in D_m^+\cap G$. Since $r_p$ is a restriction of $p$, it is weaker in the forcing order; as $G$ is upward-closed and $p\in G$, we have $r_p\in G$. Therefore by the definition of $\dot x_{\mathfrak S}$ one has $m\in \dot y[G]$. Conversely, if $m\in \dot y[G]$, then some $r_p\in \mathcal C_m\cap G$. But $r_p\forces \check m\in\dot x$, so $m\in \dot x[G]$. Thus $\dot x$ and $\dot y$ have the same members in every generic extension, and therefore $\forces \dot x=\dot x_{\mathfrak S}$.
\end{proof}

\begin{theorem}[Two-layer coding consequence]\label{thm:twolayer}
Every rank-$1$ hereditarily symmetric real in $\Ncas$ admits a two-layer code by a countable $\rho$-closed support together with a countable family of exact finite cascade packets over that support.
\end{theorem}

\begin{proof}
By \Cref{thm:rank1norm}, every rank-$1$ hereditarily symmetric real is represented by a scheme $\mathfrak S=(A,\langle\mathcal C_m:m<\omega\rangle)$ where $A\subseteq\omegaone$ is countable and $\rho$-closed and each $\mathcal C_m$ is a countable family of exact finite cascade packets over $A$.

Such a scheme has two layers of data. The first layer is the countable support $A$. The second layer is the countable family of finite packets carried by that support. In the metatheory, the class of all finite $\rho$-closed packets is well-orderable, so after fixing an ordinal $\eta$ coding that base, each $\mathcal C_m$ is coded by a subset of $\eta$, and the whole countable sequence $\langle\mathcal C_m:m<\omega\rangle$ is coded by a subset of $\Pow(\eta)$. Thus every rank-$1$ real admits a two-layer code consisting of a countable $\rho$-closed support together with a countable family of exact finite packets over that support.
\end{proof}

\section{Even exact-cardinality profile}\label{sec:even}

The one-step cascade geometry is natively binary. The complement-pair construction uses this shielding pattern to obstruct selectors on $2$-element families, and the divisibility lemma then propagates the failure to every even exact-cardinality principle.

The transition from \Cref{sec:starspan,sec:rank1} to the present section is conceptual as well as formal. The finite star-span lemma identifies the local degrees of freedom available off a fixed support, and the rank-$1$ packet calculus shows that these degrees of freedom can be compressed to finite $\rho$-closed windows. The same local binary geometry is exactly what drives the complement-pair obstruction below: a fresh one-step cascade toggle can preserve the chosen support while exchanging the two semantic values attached to one pair.

Fix $i<\omega$. For distinct $\beta,\gamma<\omegaone$, let
\[
E^i_{\beta,\gamma}=\{n<\omega:c_{\beta,i}(n)=c_{\gamma,i}(n)\},
\qquad
X^i_{\beta,\gamma}=[E^i_{\beta,\gamma}]_{\Fin}\in\Pow(\omega)/\Fin,
\]
and define the complement-pair
\[
Q^i_{\beta,\gamma}=\{X^i_{\beta,\gamma},\neg X^i_{\beta,\gamma}\}.
\]
Let
\[
\mathcal B_i=\{Q^i_{\beta,\gamma}:\beta,\gamma<\omegaone,\ \beta\neq\gamma\}.
\]

\begin{proposition}\label{prop:Bi}
The family $\mathcal B_i$ is a hereditarily symmetric family of $2$-element sets.
\end{proposition}

\begin{proof}
For each $\beta\neq\gamma$, the class $X^i_{\beta,\gamma}$ is a legitimate element of $\Pow(\omega)/\Fin$, and $X^i_{\beta,\gamma}\neq \neg X^i_{\beta,\gamma}$: otherwise the set $E^i_{\beta,\gamma}$ would be equal mod finite to its complement, impossible because the two partition $\omega$. Hence each $Q^i_{\beta,\gamma}$ has exact size $2$.

The class $X^i_{\beta,\gamma}$ depends only on the two rows $(\beta,i)$ and $(\gamma,i)$, so it has a countable $\rho$-closed support, for instance $\clrho(\{\beta,\gamma\})$. The same is true of its complement and of the unordered pair $Q^i_{\beta,\gamma}$.

It remains to verify that every basic cascade automorphism sends $X^i_{\beta,\gamma}$ either to itself or to its Boolean complement in $\Pow(\omega)/\Fin$. Let $\tau=\tau^{\mathrm{1cas}}_{\xi,j,s}$ be a basic generator. If $j\neq i$, then $\tau$ does not affect either row $(\beta,i)$ or $(\gamma,i)$, so $E^i_{\beta,\gamma}$ is unchanged. Assume now $j=i$. There are three cases.
\begin{enumerate}[label=\textup{(\arabic*)},leftmargin=2.4em]
    \item Neither $(\beta,i)$ nor $(\gamma,i)$ is toggled by $\tau$. Then $E^i_{\beta,\gamma}$ is unchanged.
    \item Both $(\beta,i)$ and $(\gamma,i)$ are toggled by $\tau$ along the same set $s$. Then equality between the two rows is preserved pointwise, so $E^i_{\beta,\gamma}$ is unchanged.
    \item Exactly one of $(\beta,i)$ and $(\gamma,i)$ is toggled by $\tau$ along $s$. Then equality is flipped exactly on $s$, so modulo $\Fin$ the class $X^i_{\beta,\gamma}$ is either unchanged (if $s$ is finite) or sent to its complement (if $s$ is cofinite, as in the applications below).
\end{enumerate}
In every case the unordered pair $Q^i_{\beta,\gamma}=\{X^i_{\beta,\gamma},\neg X^i_{\beta,\gamma}\}$ is preserved setwise. Therefore the whole family $\mathcal B_i$ is hereditarily symmetric.
\end{proof}

\begin{theorem}[No selector on the complement-pair family]\label{thm:no-selector-Bi}
There is no hereditarily symmetric choice function on $\mathcal B_i$.
\end{theorem}

\begin{proof}
Suppose toward a contradiction that $p\in\mathbb P_1$ forces that $\dot f$ is a choice function on $\mathcal B_i$, and let $A\subseteq\omegaone$ be a countable $\rho$-closed support for $\dot f$. By \Cref{lem:fresh-separation}, choose fresh nodes $\beta,\gamma\notin A$ such that $\beta\neq\gamma$, $\gamma\notin\Succ_\rho(\beta)$, and no immediate successor of $\beta$ lies in $A$. Let
\[
Q=Q^i_{\beta,\gamma}\in\mathcal B_i.
\]
Because $p$ forces that $\dot f$ is defined on the whole family, some condition $q\le p$ decides the value of $\dot f(Q)$.

Let $C=C(q,\beta,i)\subseteq\omega$ be as in \Cref{lem:finite-shielding}, choose a cofinite ground-model set $s\subseteq\omega\setminus C$, and set
\[
\tau=\tau^{\mathrm{1cas}}_{\beta,i,s}.
\]
Then \Cref{lem:finite-shielding} gives $\tau(q)=q$, while the freshness of $\beta$ implies $\tau\in\Fix^{\mathrm{1cas}}_{\rho}(A)$, so $\tau(\dot f)=\dot f$.

Now $\tau$ toggles the row $(\beta,i)$ on a cofinite set and leaves the row $(\gamma,i)$ untouched. Thus $\tau$ sends $X^i_{\beta,\gamma}$ to its complement in $\Pow(\omega)/\Fin$, so $\tau(Q)=Q$ while the two members of $Q$ are exchanged. If $q\forces \dot f(Q)=X^i_{\beta,\gamma}$, then applying $\tau$ gives
\[
q=\tau(q)\forces \dot f(Q)=\tau(X^i_{\beta,\gamma})=\neg X^i_{\beta,\gamma},
\]
contradicting that $q$ already decided the value of $\dot f(Q)$. The same contradiction arises if $q\forces \dot f(Q)=\neg X^i_{\beta,\gamma}$. Therefore no supported selector name on $\mathcal B_i$ exists.
\end{proof}

\begin{proposition}[Divisibility propagation]\label{prop:div}
For positive integers $m,k$, one has in $\ZF$:
\[
\mathrm{C}_{mk}\Rightarrow \mathrm{C}_m.
\]
Hence $\neg\mathrm{C}_m\Rightarrow \neg\mathrm{C}_{mk}$. In particular,
\[
\neg\Ctwo\Rightarrow \neg\Cn
\qquad\text{for every even }n.
\]
\end{proposition}

\begin{proof}
Let $\{A_t:t\in I\}$ be a family of nonempty $m$-element sets. Fix the set $[k]=\{0,\dots,k-1\}$ and form the family
\[
B_t=A_t\times [k]\qquad (t\in I).
\]
Each $B_t$ has size $mk$. A choice function on the family $\{B_t:t\in I\}$ yields elements $(a_t,j_t)\in B_t$, and then projection to the first coordinate gives a choice function $t\mapsto a_t$ on $\{A_t:t\in I\}$. Thus $\mathrm{C}_{mk}\Rightarrow \mathrm{C}_m$, and the remaining statements follow by contraposition.
\end{proof}

\begin{theorem}[Even exact-cardinality theorem bank]\label{thm:even-bank}
In $\Ncas$,
\[
\neg\Ctwo,\qquad \neg\ACfin,
\qquad\text{and}\qquad
\neg\Cn\ \text{for every even }n\geq 2.
\]
\end{theorem}

\begin{proof}
By \Cref{prop:Bi,thm:no-selector-Bi}, the family $\mathcal B_i$ is a hereditarily symmetric family of nonempty $2$-element sets and has no hereditarily symmetric selector. Thus $\Ncas\models\neg\Ctwo$.

Since every family of $2$-element sets is, in particular, a family of nonempty finite sets, $\neg\Ctwo$ implies $\neg\ACfin$. Finally, \Cref{prop:div} shows in pure $\ZF$ that $\mathrm{C}_{mk}\Rightarrow \mathrm{C}_m$. Applying the contrapositive with $m=2$, we obtain $\neg\Cn$ for every even $n\ge 2$.
\end{proof}

\section{Odd-side structural boundary}\label{sec:odd}

Once the even exact-cardinality profile is in hand, the first unresolved case is $\Cthree$. The point of this section is to record the strongest boundary statements currently supported by the bounded packet calculus. The one-step cascade automorphisms acting on any fixed finite $\rho$-window $K$ generate an elementary abelian $2$-group: the star generators act by $\mathbb F_2$-addition on $\mathbb F_2^K$. The following three propositions explain why this binary structure is the natural boundary of the current method.

\begin{proposition}[Odd-cardinality fixed-point principle for local $2$-group actions]\label{prop:odd-fixed}
If a finite $2$-group $G$ acts on a finite set $S$ of odd cardinality, then $S$ contains a $G$-fixed point. In particular, a support-fixed local mechanism factoring through a finite $2$-group cannot, by itself, force a contradiction to $\Cthree$.
\end{proposition}

\begin{proof}
Every $G$-orbit has cardinality a power of $2$ by the orbit-stabilizer theorem. Since $|S|$ is odd, at least one orbit has odd size, hence size $1$. This yields a $G$-fixed point.
\end{proof}

\begin{proposition}[Bounded support-equivariant semantic quotients stay dyadic]\label{prop:dyadic-quotients}
Let $Q=a+V$ be a finite local orbit of an elementary abelian $2$-group acting by translations, and let $\pi:Q\to X$ be a support-equivariant quotient. Then $|X|$ is a power of $2$. In particular, no such bounded support-local quotient yields exactly three semantic values.
\end{proposition}

\begin{proof}
Write $Q=a+V$ where $V$ is a finite-dimensional $\Ftwo$-vector space acting by translations. Let $\sim$ be the equivalence relation on $Q$ induced by $\pi$:
\[
q\sim q' \iff \pi(q)=\pi(q').
\]
Support-equivariance means exactly that $\sim$ is translation-invariant under $V$: if $q\sim q'$, then $q+v\sim q'+v$ for every $v\in V$.

Fix $q_0\in Q$ and let
\[
W=\{v\in V:q_0+v\sim q_0\}.
\]
Because $\sim$ is translation-invariant, $W$ is a linear subspace of $V$: it contains $0$, and if $v,w\in W$ then
\[
q_0+(v+w)=(q_0+v)+w\sim q_0+w\sim q_0.
\]
Now for any $q,q'\in Q$, write $q'=q+v$ for the unique $v\in V$. Then
\[
q\sim q' \iff v\in W.
\]
So the equivalence classes are exactly the affine cosets of $W$ in $Q$, each class has size $|W|$, and the quotient size is
\[
|X|=[V:W],
\]
a power of $2$. In particular, $X$ cannot have size $3$.
\end{proof}

A \emph{trace window} for the present analysis is a finite $\rho$-window $W$ viewed as the domain on which packets are evaluated. Given a $3$-element family $F=\{x_0,x_1,x_2\}$ whose elements are determined by finite packet data inside $W$, the \emph{trace profile} of $x_j$ in $W$ is the finite subset $P_j\subseteq W$ recording the nodes at which the packet data defining $x_j$ is supported. We say that $F$ is \emph{trace-separated in $W$} if the profiles $P_0,P_1,P_2$ are pairwise distinct, and that $F$ is \emph{bounded} if it is presented inside a single finite trace window. In this situation the selector below is witnessed by a hereditarily symmetric $\mathbb P_1$-name $\dot f$ supported by $W$: in each generic extension, $\dot f$ evaluates to the unique $x_j$ whose profile code $P_j$ has lexicographically minimal characteristic function in the fixed ordering of $W$. Since that comparison is definable from $W$ and all packets in the presentation have node-support contained in $W$, the name $\dot f$ is supported by $W$ and is hereditarily symmetric.

\begin{proposition}[Trace-separated bounded rigid ternary families are canonically selectable]\label{prop:canonical-selector}
Let $F=\{x_0,x_1,x_2\}$ be a bounded rigid $3$-element family, meaning that no automorphism of $\Gcas$ permutes the elements of $F$ nontrivially while fixing its support, presented inside a fixed finite trace window $W$. Suppose that each $x_j$ determines a finite trace profile $P_j\subseteq W$, and that the three profiles are pairwise distinct. Then $F$ carries a support-definable canonical selector obtained by lexicographically minimizing the profile codes.
\end{proposition}

\begin{proof}
Because $W$ is finite, each profile $P_j\subseteq W$ is finite and can be coded by its characteristic function in the fixed lexicographic order on $W$. These binary profile codes are support-definable from the presentation inside $W$. By hypothesis the three codes are pairwise distinct, so one of them is strictly lexicographically least. Selecting the corresponding $x_j$ yields a support-definable canonical selector on $F$.
\end{proof}

\begin{remark}
Propositions~\ref{prop:odd-fixed}, \ref{prop:dyadic-quotients}, and \ref{prop:canonical-selector} do \emph{not} prove $\Cthree$, and they do \emph{not} prove $\neg\Cthree$. They show only that the obvious local binary routes are exhausted inside the present bounded packet calculus. The third proposition is intentionally conditional: what is ruled out at present are bounded rigid ternary presentations that are already trace-separated inside one finite window. Any future route to $\neg\Cthree$ appears to require an unbounded semantic quotient, a genuinely nonlocal obstruction, or some new mechanism that breaks the dyadic character of the current analysis.
\end{remark}

\section{Open problems}\label{sec:open}

The present paper isolates a new one-step cascade geometry and settles its even exact-cardinality profile, but several natural structural questions remain open.

\begin{enumerate}[label=\textup{(\arabic*)},leftmargin=2.4em]
    \item \textbf{The odd exact-cardinality profile.} The first open case is whether $\Ncas\models \neg\Cthree$. The results of \Cref{sec:odd} show that the current bounded binary packet calculus does not decide this question, but they do not rule out either a deeper obstruction or a positive odd-cardinality phenomenon.
    \item \textbf{Higher-rank normalization.} \Cref{thm:rank1norm} gives a packet normalization theorem for rank-$1$ hereditarily symmetric reals. Does the cascade geometry admit an analogous normalization scheme for higher-rank names, or is rank $1$ the natural endpoint of the present finite star-span analysis?
    \item \textbf{Strength of low-rank coding.} The two-layer coding consequence of \Cref{thm:twolayer} suggests that the cascade geometry may organize low-rank information more tightly than geometries whose automorphism groups remain rowwise local and have uniformly countable total support. Is there a precise low-rank coding or Kinna--Wagner separation theorem that witnesses a structural distinction of that kind?
\end{enumerate}

\section{Relation to previous work}\label{sec:related}

The present construction sits in four overlapping literatures.

First, at the level of symmetric-extension architecture, the paper should be read against Grigorieff's analysis of intermediate submodels, Karagila's framework for iterating symmetric extensions, and the recent reformulation of symmetric systems, quotients, completions, and equivalences in \cite{Grigorieff1975,Karagila2019Iterating,KaragilaSchilhan2026Theory}. The novelty here is not a new general framework, but a specific one-step symmetry geometry whose local combinatorics are governed by $\rho$-closed windows and one-step stars.

Second, although the present paper is not an iteration paper, its support-local transport arguments are conceptually close to the upwards-homogeneity line of work: one isolates hypotheses under which higher-stage symmetry does not create genuinely new subsets of an earlier model. Our $\rho$-closed support reductions play a different role and are proved directly in one step, but they belong to the same family of questions about how much of a name is already visible on a controlled support; compare \cite{RyanSmithSchilhanWei2026Upwards}.

Third, the anti-$\BPI$ conclusion should be compared with both the classical symmetric-model landscape and the recent dynamical analysis of $\BPI$. Classical Cohen-type and Mostowski-type examples show that choicelessness can arise from highly permutation-driven geometries; see the surveys in \cite{Jech1973AC,Felgner1971}. By contrast, Ransom's recent work studies sufficient dynamical conditions --- via the filter extension property and related Ramsey-type hypotheses --- under which $\BPI$ \emph{holds} in symmetric extensions \cite{Ransom2025BPI}. The present paper does not engage the filter extension property directly. Its contribution is instead a direct failure-of-$\BPI$ theorem in a different one-step geometry, extracted from the same fresh-support shielding mechanism that later drives the even exact-cardinality profile.

Fourth, the packet-normalization and coding consequences place the model near the Kinna--Wagner/intermediate-model interface studied in \cite{KaragilaSchilhan2026IMKWP,HayutShani2024}. The even exact-cardinality conclusions should be viewed in the standard weak-choice landscape recorded in \cite{Jech1973AC,Jech2003,HowardRubin1998}. What is specific to the present paper is that the same local binary geometry controls both sides of the argument: finite star-spans and packet normalization explain how rank-$1$ information compresses to finite $\rho$-windows, while fresh support outside a fixed window gives the complement-pair obstruction to choice on the even side.

\bibliographystyle{alpha}
\bibliography{cascade_model_refs_v15}

\end{document}